\documentclass[12pt]{amsart}
\usepackage[margin=1.2in]{geometry}

\usepackage{enumitem}
\usepackage{pgfplots}
\usepackage{amsxtra,amscd}
\usepackage{xcolor}

\usepackage{graphicx}
\usepackage{multirow}
\usepackage{amsmath,amssymb,amsfonts}
\usepackage{amsthm}
\usepackage{mathrsfs}
\usepackage[title]{appendix}
\usepackage{xcolor}

\usepackage{todonotes}
\usepackage{csquotes}%
\usepackage{hyperref}
\usepackage[foot]{amsaddr}

\newcommand{\R}{\mathbb{R}}
\newcommand{\N}{\mathbb{N}}
\newcommand{\ind}{\operatorname{ind}}

\newtheorem{theorem}{Theorem}
\newtheorem{proposition}[theorem]{Proposition}
\newtheorem{fact}[theorem]{Fact}
\newtheorem{lem}[theorem]{Lemma}
\newtheorem{cor}[theorem]{Corollary}

\theoremstyle{definition}%
\newtheorem{remark}[theorem]{Remark}%
\newtheorem{example}[theorem]{Example}%

\begin{document}

\title[Hypercyclic and mixing  composition operators on $\mathscr{O}_M(\R)$]{Hypercyclic and mixing composition operators on $\mathscr{O}_M(\R)$}

\author[T.\ Kalmes]{Thomas Kalmes$^1$}
\thanks{This version of the article has been accepted for publication, after peer review (when applicable) but is not the Version of Record and does not reflect post-acceptance improvements, or any corrections. The Version of Record is available online at: \href{http://dx.doi.org/10.1007/s13398-024-01649-1}{http://dx.doi.org/10.1007/s13398-024-01649-1}}
\address{$^1$Faculty of Mathematics, Chemnitz University of Technology, 09107 Chemnitz, Germany}
\email{thomas.kalmes@math.tu-chemnitz.de}

\author[A. Przestacki]{Adam Przestacki$^2$}
\address{$^2$Faculty of Mathematics and Computer Science, Adam Mickiewicz University, Uniwersytetu Pozna\'{n}skiego 4, 61-614 Pozna\'{n}, Poland}
\email{adamp@amu.edu.pl}

\begin{abstract}
In this paper we characterize mixing composition operators acting on the space $\mathscr{O}_M(\R)$ of slowly increasing smooth functions. Moreover we relate the mixing property of those operators with the solvability of Abel's functional  equation and we give a sufficient condition for sequential hypercyclicity of composition operators on $\mathscr{O}_M(\R)$. This is used to prove that many mixing composition operators are hypercyclic.\\

\noindent Keywords: composition operator, linear dynamics, mixing operator, Abel's equation, hypercyclic operator\\

\noindent MSC 2020: 47B33, 47A16, 46E10
\end{abstract}

\maketitle

\section{Introduction}

The study of dynamical properties of (continuous linear) operators $T\in L(E)$ on topological vector spaces $E$ has attracted much interest in recent years. While there are few articles dealing with dynamics of operators on non-metrizable spaces, the vast majority of contributions concentrates on the dynamics of operators defined on separable Fr\'echet spaces. The advantage of completeness and metrizability stems from the applicability of Baire category arguments which are a powerful tool in this context. One prominent example of such a tool is Birkhoff's transitivity criterion, stating that every topologically transitive operator on a separable Fr\'echet space is hypercyclic. Recall that $T$ is said to be \textit{topologically transitive} if for every pair of non-empty, open subsets $U,V$ of $E$ it holds $T^n(U)\cap V\neq\emptyset$ for some $n\in\N$, while $T$ is \textit{(sequentially) hypercyclic} whenever there is $x\in E$ whose orbit $\{T^nx;\,n\in\N_0\}$ under $T$ is (sequentially) dense in $E$. Clearly, on arbitrary Hausdorff topological vector spaces, every hypercyclic operator is topologically transitive. Moreover, $T$ is called \textit{(topologically) mixing} if for every pair of non-empty, open subsets $U,V$ of $E$ it holds $T^n(U)\cap V\neq\emptyset$ for all sufficiently large $n\in\N$, while $T$ is said to be \textit{chaotic} if it is topologically transitive and if the set of periodic points of $T$ is dense in $E$. In particular, on Fr\'echet spaces, mixing operators are sequentially hypercyclic. 

The aim of this paper is to study dynamical properties of composition operators acting on the space $\mathscr{O}_M(\R)$ of multipliers of the space of rapidly decreasing, smooth functions $\mathscr{S}(\R)$ on $\R$. More precisely, we are interested in (sequential) hypercyclicity and mixing of composition operators on $\mathscr{O}_M(\R)$. Recall that 
$\mathscr{O}_M(\R)$ is given by 
\[
\mathscr{O}_M(\R)=\cap_{m=1}^\infty\cup_{n=1}^\infty\mathscr{O}_n^m(\R),
\]
where 
\[
\mathscr{O}_n^m(\R):=\left\lbrace f\in C^m(\R):\,|f|_{m,n}:=\sup_{x\in\R, 0\leq j\leq m}(1+|x|^2)^{-n}|f^{(j)}(x)|<\infty\right\rbrace.
\]
The space $\mathscr{O}_M(\R)$ is equipped with a natural locally convex topology which makes it a complete, ultrabornological, non-metrizable locally convex space. Hence, $\mathscr{O}_M(\R)$ is not a Fr\'echet space and thus, a mixing operator on $\mathscr{O}_M(\R)$ need not be (sequentially) hypercyclic. The study of dynamical properties of composition operators on $\mathscr{O}_M(\R)$ was initiated by Albanese, Jord\'{a} and Mele. In \cite{MR4419539}, among  other things, they showed that a composition operator $C_\psi:\mathscr{O}_M(\R)\rightarrow\mathscr{O}_M(\R), f\mapsto f\circ\psi$ with a smooth symbol $\psi:\R\rightarrow\R$ is correctly defined (and hence continuous by a standard application of De Wilde's Closed Graph Theorem) if and only if $\psi\in\mathscr{O}_M(\R)$. Moreover, they studied dynamical properties (power boundedness, mean ergodicity) of those operators and showed that the translation operator is mixing on $\mathscr{O}_M(\R)$.

Composition operators play an important role in functional analysis. Their dynamical properties on various spaces of functions and sequences were intensively studied over the past decades by many authors, see \cite{MR3199721,  MR2496409, shapiro, MR3147716, MR2735001,  MR2322758, MR4358440} for (weighted) composition operators on spaces of holomorphic functions, \cite{MR2990627} for composition operators on spaces of analytic functions, \cite{MR3543813} for weighted composition operators on spaces of smooth functions, \cite{MR3999461} for composition operators on spaces of functions defined by local properties, \cite{MR4047689} for weighted translation operators acting on the Schwartz space, \cite{MR2358980} for weighted composition operators on $L^p$-spaces and weighted spaces of continuous functions, and \cite{MR2990627,MR3892306,MR4310296} for hypercyclicity results for non-metrizable locally convex spaces, as well as the references therein.

It is the purpose of this note to complement the results from \cite{MR4419539}. In particular, in Theorems \ref{thm: mixing_bij} and \ref{thm:mix_not_bij} we characterize mixing composition operators $C_\psi$ on $\mathscr{O}_M(\R)$ in terms of their symbol $\psi$. Moreover we show in Theorem \ref{thm: abel} that this property is closely related to the solvability in $\mathscr{O}_M(\R)$ of Abel's functional equation, i.e.\ the problem to find for a given symbol $\psi\in\mathscr{O}_M(\R)$ a function $H\in\mathscr{O}_M(\R)$ which satisfies the equation
\[
H(\psi(x))=H(x)+1.
\]
Additionally, we give a sufficient condition on the symbol $\psi$ of a composition operator $C_\psi$ on $\mathscr{O}_M(\R)$ to be (sequentially) hypercyclic, see Theorem \ref{theorem: sufficient for hypercyclicity}. This condition allows to identify the translation operator to be (sequentially) hypercyclic on $\mathscr{O}_M(\R)$.
Moreover, thanks to Theorem \ref{thm: abel}, we  deduce in  Corollary \ref{cor: hypercyclic} that many mixing composition operators are (sequentially) hypercyclic.

It has been shown by the second author in \cite[Theorem 4.2]{MR3543813} that  a composition operator $C_\psi$ on the Fr\'echet space of smooth functions $C^\infty(\R)$ is hypercyclic if and only if it is mixing if and only if  $\psi$ has a non-vanishing derivative and no fixed points. Applying standard arguments (see Proposition \ref{proposition: necessary for transitivity}), it is easily seen that for a symbol $\psi\in\mathscr{O}_M(\R)$ with topologically transitive $C_\psi$ on $\mathscr{O}_M(\R)$, the corresponding composition operator on the space of smooth functions $C^\infty(\R)$ is topologically transitive as well. We give an example (see Example \ref{ex: not_mi}) that the converse implication is not true.

The paper is organized as follows. In section \ref{section: hypercyclicity}, after recalling some topological properties of $\mathscr{O}_M(\R)$ which will be relevant for our purpose, we study (sequential) hypercyclicity of composition operators on $\mathscr{O}_M(\R)$. In section \ref{section: mixing} we characterize mixing  composition operators in terms of their symbol while in section \ref{section: Abel} we investigate the connection between mixing and the solvability of Abel's equation in $\mathscr{O}_M(\R)$.  

Through the paper, by $\N=\{1,2,3,\ldots\}$ we denote the set of natural numbers. 
For a function $\psi:\R\to \R$
\begin{itemize}
	\item we define $\psi_0:\R\to\R$ as $\psi_0(x)=x$,
	\item for every $n\in \N$ we define $\psi_n:\R\to\R$
	inductively via the  formula $\psi_n(x)=\psi(\psi_{n-1}(x))$,
	\item whenever $\psi$ is injective, for every 
	$n\in\mathbb{Z}\backslash(\mathbb{N}\cup\{0\})$ we define $\psi_n:\psi_{-n}(\R)\to
	\R$ via the rule: $\psi_n(x)=y$ if and only if
	$\psi_{-n}(y)=x$. 
\end{itemize} 
Finally, for further reference, let us recall Fa\`{a} di Bruno's formula which states for smooth functions $f,g$ and $j\in \N$
\[
(f\circ g)^{(j)}(x)=\sum_{\substack{i_1,i_2,\ldots, i_j\geq 0 \\ i_1+2i_2+\cdots +ji_j=j}}
\frac{j!}{i_1!i_2!\cdots i_j!}f^{(i_1+i_2+\cdots +i_j)}(g(x))
\cdot
\prod_{r=1}^j\left(\frac{g^{(r)}(x)}{r!}\right)^{i_r}.
\]
For the definition of hypercyclicity, mixing and other unexplained notions from linear dynamics we refer to \cite{MR2919812}, while we refer to \cite{MR1483073} for anything related to functional analysis.

\section{Hypercyclicity of composition operators}\label{section: hypercyclicity}

Obviously, on$$\mathscr{O}_n^m(\R)=\left\lbrace f\in C^m(\R):\,|f|_{m,n}:=\sup_{x\in\R, 0\leq j\leq m}(1+|x|^2)^{-n}|f^{(j)}(x)|<\infty\right\rbrace$$
a norm is given by $|\cdot|_{m,n}$ and equipped with this norm, $\mathscr{O}_n^m(\R)$ is a Banach space, $m,n\in\N$. Additionally, $\mathscr{O}^m(\R):=\ind_{n\rightarrow\infty}\mathscr{O}_n^m(\R)$ is a complete (LB)-space
. The space  $\mathscr{O}_M(\R)$ is endowed  with its natural locally convex topology, i.e. $\mathscr{O}_M(\R)$ is the projective limit of the (LB)-spaces $\mathscr{O}^m(\R)$, $m\in\N$, where the linking maps from $\mathscr{O}^{m+1}(\R)$ to $\mathscr{O}^m(\R)$ are the inclusions. 

A fundamental system of continuous seminorms on $\mathscr{O}_M(\R)$ is given by
\[
p_{m,v}(f)=\sup_{x\in\R}\max_{0\leq j\leq m}|v(x) f^{(j)}(x)|,\, f\in\mathscr{O}_M(\R), m\geq 0, v\in\mathscr{S}(\R),
\]
where 
$\mathscr{S}(\R)$ is the space of rapidly  decreasing smooth functions (see  \cite{MR0205028}).
In fact  it is not difficult to see that $\mathscr{O}_M(\R)$
is the space of smooth functions $f$ on $\R$ such that $p_{m,v}(f)$ is finite for every $m\geq0$  and  $v\in\mathscr{S}(\R)$, as well as the space of multipliers of  $\mathscr{S}(\R)$. Obviously, $\mathscr{O}_M(\R)$ embeds continuously into $C^\infty(\R)$, and, as is well known, the space of compactly supported smooth functions $\mathscr{D}(\R)$ is dense in $\mathscr{O}_M(\R)$. Consequently, $\mathscr{O}_M(\R)$ is dense in $C^\infty(\R)$. Below we will use the following property of the topology of $\mathscr{O}_M(\mathbb{R})$ (see \cite[Remark  2.2]{MR4419539}).
\begin{fact}
	\label{f: convergence}
	A sequence $(f_n)_{n\in \N}$ of functions from $\mathscr{O}_M(\mathbb{R})$ is convergent to $f$ in  $\mathscr{O}_M(\mathbb{R})$ if and only if $(f_n)_{n\in \N}$ is bounded in $\mathscr{O}_M(\mathbb{R})$ and converges to $f$ in $C^\infty(\mathbb{R})$.
\end{fact}

Recently, it was shown in \cite{MR4419539} that the translation operator on $\mathscr{O}_M(\mathbb{R})$, i.e.\ $C_\psi\colon\mathscr{O}_M(\mathbb{R})\to \mathscr{O}_M(\mathbb{R})$, $f\mapsto f \circ \psi$ with $\psi(x)=x+1$, is mixing. As already mentioned in the introduction, $\mathscr{O}_M(\mathbb{R})$ is not a Fr\'{e}chet space and thus, Birkhoff's Transitivity Theorem cannot be applied to conclude hypercyclicity of the translation operator. The main objective of this section is to prove that the translation operator is indeed hypercyclic on $\mathscr{O}_M(\mathbb{R})$.

We start with the following trivial observation.

\quad

\begin{proposition}\label{proposition: necessary for transitivity}
	For $\psi\in\mathscr{O}_M(\R)$ consider the following conditions. 
	\begin{itemize}
		\item [(i)] The composition operator $C_\psi:\mathscr{O}_M(\R)\rightarrow\mathscr{O}_M(\R)$ is topologically transitive.
		\item[(ii)] The composition operator $\tilde{C}_\psi:C^\infty(\R)\rightarrow C^\infty(\R), f\mapsto f\circ \psi$ is topologically transitive.
		\item[(iii)] $\psi$ has no fixed points and $\psi'(x)>0$ for every $x\in\R$.
	\end{itemize}
	Then, (i) implies (ii), while (ii) and (iii) are equivalent.
\end{proposition}

\begin{proof}
	Since the inclusion $i:\mathscr{O}_M(\R)\rightarrow C^\infty(\R), f\mapsto f$ is continuous and has dense range, topological transitivity of $C_\psi$ implies the topological transitivity of $\tilde{C}_\psi$ (see, \cite[Proposition 1.13]{MR2919812}). Thus, by \cite[Theorem 4.2]{MR3543813}, $\psi$ has no fixed points and $\psi'(x)\neq 0$ for every $x\in\R$. Hence, we either have $\psi'(x)>0$ for every $x\in\R$ or $\psi'(x)<0$ for each $x\in\R$. Since the latter condition contradicts that $\psi$ has no fixed points, (ii) implies (iii). Another application of \cite[Theorem 4.2]{MR3543813} shows that (iii) implies (ii).
\end{proof}

In contrast to composition operators on the Fr\'echet space $C^\infty(\R)$, the conditions in (i) of the previous Proposition are only necessary but not sufficient for topological transitivity of a composition operator on $\mathscr{O}(\R)$, as will be shown in Example \ref{ex: not_mi} below.

\quad

\begin{theorem}\label{theorem: sufficient for hypercyclicity}
	Let $\psi\in\mathscr{O}_M(\mathbb{R})$ be bijective such that $\psi(x)>x$ as well as $\psi'(x)>0$ for every $x\in\mathbb{R}$. Additionally, assume that
	\begin{equation}\label{eq:hypothesis 1 on psi}
		\forall\,j\in\mathbb{N}\,\exists\,\beta_j\in\mathbb{R}, C_j>0, t_j\in\mathbb{N}\,\forall\,x\in (\beta_j,\infty), n\in\mathbb{N}:\,|(\psi_{-n})^{(j)}(x)|\leq C_j(1+|x|^2)^{t_j}
	\end{equation}
	and
	\begin{equation}\label{eq:hypothesis 2 on psi}
		\forall\,j\in\mathbb{N}\,\exists\,\alpha_j\in\mathbb{R}, C_j>0, t_j\in\mathbb{N}\,\forall\,x\in (-\infty,\alpha_j), n\in\mathbb{N}:\,|(\psi_n)^{(j)}(x)|\leq C_j(1+|x|^2)^{t_j}.
	\end{equation}
	Then, $C_{\psi}:\mathscr{O}_M(\mathbb{R})\rightarrow\mathscr{O}_M(\mathbb{R})$ is sequentially hypercyclic.
\end{theorem}

\quad

Before we prove Theorem \ref{theorem: sufficient for hypercyclicity} we make the following comment.

\begin{remark}
	For bijective $\psi\in\mathscr{O}_M(\R)$ without fixed points, it either holds $\psi(x)>x$ for every $x\in\R$ or $\psi(x)<x$ for each $x\in\R$. While Theorem \ref{theorem: sufficient for hypercyclicity} deals with the first case, replacing in hypothesis \eqref{eq:hypothesis 1 on psi} \enquote{$\forall x\in (\beta_j,\infty)$} by \enquote{$\forall x\in (-\infty,\beta_j)$} and in hypothesis \eqref{eq:hypothesis 2 on psi} \enquote{$\forall x\in (-\infty,\alpha_j)$} by \enquote{$\forall x\in (\alpha_j,\infty)$} gives an analogous result for the case $\psi(x)<x, x\in\R$. Indeed, let $r(x)=-x$, $x\in\mathbb{R}$, be the reflection at the origin. Then, $C_r$ is bijective on $\mathscr{O}_M(\mathbb{R})$ with $C_r^2=\operatorname{id}_{\mathscr{O}_M(\mathbb{R})}$. Additionally, for $\psi\in\mathscr{O}_M(\mathbb{R})$ we set $\sigma(\psi)=-C_r(\psi)$ so that $\sigma(\sigma(\psi))=\psi$. Then, we have $C_\psi=C_r\circ C_{\sigma(\psi)}\circ C_r$, so that $C_\psi$ and $C_{\sigma(\psi)}$ are conjugate, in particular, $C_\psi$ is (sequentially) hypercyclic if and only if $C_{\sigma(\psi)}$ is. Obviously, $\psi(x)<x$ for every $x\in\mathbb{R}$ precisely when $\sigma(\psi)(x)>x$ for every $x\in\mathbb{R}$. Additionally, $\psi$ is bijective if and only if $\sigma(\psi)$ is bijective, and $\left(\sigma(\psi)\right)^{(j)}(x)=(-1)^{j-1}\psi^{(j)}(-x)=(-1)^j\sigma(\psi^{(j)})(x)$.
\end{remark}

\begin{proof}[Proof of Theorem \ref{theorem: sufficient for hypercyclicity}]
	We will explicitly construct a function $g\in\mathscr{O}_M(\mathbb{R})$ whose orbit under $C_\psi$ is sequentially dense in $\mathscr{O}_M(\mathbb{R})$. In order to do so, let $(p_n)_{n\in\mathbb{N}}$ be a sequence of compactly supported smooth functions on $\mathbb{R}$ such that $\{p_n: n\in\mathbb{N}\}$ is dense in $\mathscr{O}_M(\mathbb{R})$ and such that for every $m\in\mathbb{N}$ there are infinitely many $n\in\mathbb{N}$ with $p_n=p_m$.
	
	Since $\psi$ is bijective, without fixed points, and $\psi(x)>x$ for every $x$, the sequence $(\psi_n(x))_{n\in\mathbb{N}}$ is strictly increasing and tends to infinity while $(\psi_{-n}(x))_{n\in\mathbb{N}}$ is strictly decreasing with $\lim_{n\rightarrow\infty}\psi_{-n}(x)=-\infty$. In particular, for every compact subset $K$ of $\mathbb{R}$ there is $N\in\mathbb{N}$ such that neither $\psi_n(K)$ nor $\psi_{-n}(K)$ intersects $K$ whenever $n\geq N$.
	
	Next, we choose a strictly increasing sequence $(k_n)_{n\in\mathbb{N}}$ of nonnegative integers by the following recursive procedure. First we choose $k_1=0$. If $k_1,\ldots, k_n$ have already been chosen, let $k_{n+1}$ be strictly larger than $k_n$ such that the following conditions are satisfied.
	\begin{enumerate}
		\item[(a)] There exists $t\in\mathbb{R}$ such that the support of $p_n\circ\psi_{-k_n}$ is contained in $(-\infty, t)$ while the support of $p_{n+1}\circ\psi_{-k_{n+1}}$ is contained in $(t,\infty)$.
		\item[(b)] For $1\leq l\leq n$ and $0\leq j\leq n+1$ the support of $p_l\circ\psi_{k_{n+1}-k_l}$ is contained in $(-\infty,\min\{-n-1,\alpha_1,\ldots,\alpha_{n+1}\})$ and
		$\left|p_l^{(j)}\left(\psi_{k_{n+1}-k_l}(x)\right)\right|\leq |x|$ for every $x\in\mathbb{R}$.
		\item[(c)] For $1\leq l\leq n$ and $0\leq j\leq n+1$ the support of $p_{n+1}\circ\psi_{-(k_{n+1}-k_l)}$ is contained in $(\max\{n+1,\beta_1,\ldots,\beta_{n+1}\},\infty)$ and $\left|p_{n+1}^{(j)}\left(\psi_{-(k_{n+1}-k_l)}(x)\right)\right|\leq |x|$ for every $x\in\mathbb{R}$.
	\end{enumerate}
	It is clear that such a choice of $k_{n+1}$ is possible.
	
	From $(a)$ it follows that the functions $p_n\circ\psi_{-k_n}$, $n\in\mathbb{N}$, have pairwise disjoint supports so that by
	$$\forall\,x\in\mathbb{R}:\, g(x)=\sum_{n=1}^\infty p_n\left(\psi_{-k_n}(x)\right)$$
	a smooth function $g$ is defined on $\mathbb{R}$. Keeping in mind that for $x\in\mathbb{R}$ at most one of the defining summands of $g$ does not vanish at $x$, an application of condition $(c)$ for $l=1$ combined with hypothesis \eqref{eq:hypothesis 1 on psi} on $\psi$, and Fa\`{a} di Bruno's formula yields for any nonnegative integer $m$ the existence of $C>0$ and $t\in\mathbb{N}$ such that for $x\in\mathbb{R}$ and $0\leq j\leq m$
	$$|g^{(j)}(x)|\leq\max\left\{\left|\left(p_n\circ\psi_{-k_n}\right)^{(j)}(y)\right|; 1\leq n\leq m,\, y\in\mathbb{R}\right\}+C(1+|x|^2)^t|x|,$$
	so that $g\in\mathscr{O}_M(\mathbb{R})$.
	
	We claim that $\{C_\psi^n(g): n\in\mathbb{N}\}$ is sequentially dense in $\mathscr{O}_M(\mathbb{R})$. To prove this it is enough to show that the set $\{p_N: N\in \mathbb{N}\}$ is contained in the sequential closure of $\{C_\psi^n(g): n\in\mathbb{N}\}$. Thus, we fix $N\in\mathbb{N}$. Let $(s_i)_{i\in\mathbb{N}}$ be a strictly increasing sequence of nonnegative integers such that $p_N=p_{s_i}$, $i\in\mathbb{N}$. Then, for $i\in\mathbb{N}$, we have
	\begin{equation}\label{eq:proof of hypercyclicity 1}
		C_\psi^{k_{s_i}}(g)-p_N=\sum_{n=1}^{s_i-1} \left(p_n\circ\psi_{k_{s_i}-k_n}\right) + \sum_{n=s_i+1}^\infty\left(p_n\circ\psi_{-(k_n-k_{s_i})}\right).
	\end{equation}
	By condition $(b)$, for $n<s_i$, the support of $p_n\circ\psi_{k_{s_i}-k_n}$ is contained in $(-\infty, -s_i)$. Likewise, for $n>s_i$, condition $(c)$ ensures that the support of $p_n\circ\psi_{-(k_n-k_{s_i})}$ is contained in $(s_i,\infty)$. Hence, both sequences of functions
	\begin{equation}\label{eq:proof of hypercyclicity 2}
		\left(\sum_{n=1}^{s_i-1} \left(p_n\circ\psi_{k_{s_i}-k_n}\right)\right)_{i\in\mathbb{N}}\quad\text{and}\quad \left(\sum_{n=s_i+1}^\infty\left(p_n\circ\psi_{-(k_n-k_{s_i})}\right)\right)_{i\in\mathbb{N}}
	\end{equation}
	converge to zero in $C^\infty(\mathbb{R})$. Thus, Fact 1 combined with \eqref{eq:proof of hypercyclicity 1} will imply $p_N=\lim_{i\rightarrow\infty}C_\psi^{k_{s_i}}(g)$ in $\mathscr{O}_M(\mathbb{R})$ once we have shown that both sequences in \eqref{eq:proof of hypercyclicity 2} are bounded in $\mathscr{O}_M(\mathbb{R})$, thereby completing the proof.
	
	Since $p_n\circ\psi_{-k_n}$, $n\in\mathbb{N}$, have mutually disjoint supports, the summands of the first sequence in \eqref{eq:proof of hypercyclicity 2} have mutually disjoint supports as do the ones of the second sequence. The same arguments which we used to prove that $g$ belongs to $\mathscr{O}_M(\mathbb{R})$ yield that the second sequence in \eqref{eq:proof of hypercyclicity 2} is bounded in $\mathscr{O}_M(\mathbb{R})$. Refering to hypothesis \eqref{eq:hypothesis 2 on psi} and to condition $(b)$ instead of hypothesis \eqref{eq:hypothesis 1 on psi} and condition $(c)$, respectively, one shows mutatis mutandis that the first sequence in \eqref{eq:proof of hypercyclicity 2} is bounded in $\mathscr{O}_M(\mathbb{R})$, too.
\end{proof}

\begin{cor}\label{corollary: sufficient for hypercyclicity}
	Let $\psi\in\mathscr{O}_M(\R)$ be bijective, without fixed points and such that $\psi'(x)>0$ for every $x\in\R$ and such that $\{(\psi_n)':\,n\in\mathbb{Z}\}$ is bounded in $\mathscr{O}_M(\R)$. Then, the composition operator $C_\psi:\mathscr{O}_M(\R)\rightarrow\mathscr{O}_M(\R)$ is sequentially hypercyclic.
\end{cor}

\begin{proof}
	Since either $\psi(x)>x$ for every $x\in\R$ or $\psi(x)<x$ for every $x\in\R$ the assertion follows immediately from Theorem \ref{theorem: sufficient for hypercyclicity} and the comment preceding its proof.
\end{proof}

\begin{cor}\label{corollary: chaos of shifts}
	For $\beta\in\mathbb{R}\backslash\{0\}$ and $\psi(x)=x+\beta$ the composition operator $C_\psi$ is sequentially hypercyclic on $\mathscr{O}_M(\mathbb{R})$. Additionally, $C_\psi$ is chaotic.
\end{cor}

\begin{proof}
	The sequential hypercyclicity of $C_\psi$ follows immediately from Corollary \ref{corollary: sufficient for hypercyclicity}. Additionally, considering the set $\{\sum_{n\in\mathbb{Z}}g(\cdot+ nl k_g \beta); g\in \mathscr{D}(\R), l\in\mathbb{N}\}$, where $k_g\in\N$ is chosen in such a way that $[\min\text{supp}\,g,\max\text{supp}\,g ]$ and $[k_g+\min\text{supp}\,g,k_g+\max\text{supp}\,g ]$ are disjoint. Then, by standard arguments, this set is dense in $\mathscr{O}_M(\R)$ and consists of periodic points for $C_\psi$. Thus, $C_\psi$ is chaotic.
\end{proof}

\section{Mixing composition operators}\label{section: mixing}
In this section we characterize mixing operators $C_\psi$ acting on 
$\mathscr{O}_M(\R)$ in terms of their symbol $\psi$.

\quad

\begin{theorem}
	\label{thm: mixing_bij}
	Let $\psi\in \mathscr{O}_M(\R)$ be surjective. Then, the  following conditions are equivalent.
	\begin{enumerate}[label=(\roman*)]
		\item The operator $C_\psi\colon\mathscr{O}_M(\mathbb{R})\to \mathscr{O}_M(\mathbb{R})$, $f\mapsto f \circ \psi$ is mixing.
		\item $\psi$ is injective with a non-vanishing derivative and without fixed points such that for every $a\in\R$ and each $k\in\mathbb{N}$, for arbitrary $v\in \mathscr{S}(\mathbb{R})$ it holds
		\[
		\lim_{n\to\infty} \sup_{x\in\psi_{-n}\left([\min\{a,\psi(a)\},\max\{a,\psi(a)\}]\right)}
		\left|v(x)(\psi_n)^{(k)}(x)\right|=0
		\]
		and
		\[
		\lim_{n\to\infty} \sup_{x\in\psi_{n}\left([\min\{a,\psi(a)\},\max\{a,\psi(a)\}]\right)}
		\left|v(x)(\psi_{-n})^{(k)}(x)\right|=0.
		\] 
		\item $\psi$ is injective with a non-vanishing derivative and without fixed points, and there are $a,b\in\R$ such that for every $k\in\mathbb{N}$ and $v\in \mathscr{S}(\mathbb{R})$ we have
		\[
		\lim_{n\to\infty} \sup_{x\in\psi_{-n}\left([\min\{a,\psi(a)\},\max\{a,\psi(a)\}]\right)}
		\left|v(x)(\psi_n)^{(k)}(x)\right|=0
		\]
		and
		\[
		\lim_{n\to\infty} \sup_{x\in\psi_{n}\left([\min\{b,\psi(b)\},\max\{b,\psi(b)\}]\right)}
		\left|v(x)(\psi_{-n})^{(k)}(x)\right|=0.
		\] 
	\end{enumerate}
\end{theorem}
\begin{proof} $ $\\
	Clearly, $(ii)$ implies $(iii)$.\\
	$(iii)\Rightarrow (i)$ Since $\psi$ does not have a fixed point, we have either $\psi(x)>x$ for all $x\in\R$, or $\psi(x)<x$ for all $x\in\R$. We only consider the case $\psi(x)>x$, the other case is treated, mutatis mutandis, with the same arguments. Thus, $\min\{a,\psi(a)\}=a, \max\{a,\psi(a)\}=\psi(a)$ and $\min\{b,\psi(b)\}=b, \max\{b,\psi(b)\}=\psi(b)$.
	
	Taking into account that $\lim_{n\rightarrow\infty}\psi_n(a)=\infty$ and $\lim_{n\rightarrow\infty}\psi_{-n}(a)=-\infty$ we have $\R=\cup_{m\in\mathbb{Z}}\left(\psi_m(a),\psi_{m+2}(a)\right)$
	and the sequence of open intervals $\left(\psi_m(a),\psi_{m+2}(a)\right)_{m\in\mathbb{Z}}$ is a locally finite cover of $\R$.
	Let $(\phi_m)_{m\in\mathbb{Z}}$ be a partition of unity on $\R$ subordinate to it.
	Likewise,
	let $(\eta_m)_{m\in\mathbb{Z}}$ be a partition of unity on $\R$ subordinate to the locally finite cover of $\R$ by the sequence of open intervals $\left((\psi_m(b),\psi_{m+2}(b))\right)_{m\in\mathbb{Z}}$.
	
	Since compactly supported functions are dense in 
	$\mathscr{O}_M(\R)$, in view of Kitai's criterion (see \cite[Thm. 12.31]{MR2919812}), it is enough to show that
	for every compactly supported smooth function $f$
	the sequences $(f\circ\psi_n)_{n\in\N}$ and $(f\circ\psi_{-n})_{n\in\N}$
	converge to zero in $\mathscr{O}_M(\R)$. Note that with $f$ also $f\circ\psi_{-n}$ is a compactly supported smooth function and thus belongs to $\mathscr{O}_M(\R)$. When considering $(f\circ\psi_n)_{n\in\N}$, respectively $(f\circ\psi_{-n})_{n\in\N}$, we may replace $f$ by $\phi_m f$ and $\eta_m f$, respectively, so that without loss of generality $\text{supp}\,f\subset (\psi_m(a),\psi_{m+2}(a))$ and $\text{supp}\,f\subset (\psi_m(b),\psi_{m+2}(b))$, respectively.
	
	We will show that the sequence $(f\circ\psi_n)_{n\in \N}$
	tends to zero in $\mathscr{O}_M(\R)$. To do this, let us fix 
	$v\in\mathscr{S}(\R)$ and $k\geq 0$ and we observe
	\begin{eqnarray}\label{eq: mixing 1}
		\sup_{x\in\mathbb{R}}\left|v(x)(f\circ\psi_n)^{(k)}(x)\right|&=&\sup_{x\in [\psi_{m-n}(a),\psi_{m+2-n}(a)]}\left|v(x)(f\circ\psi_n)^{(k)}(x)\right|\nonumber\\
		&\leq&\sup_{x\in\psi_{-(n-m)}\left( [a,\psi(a)]\right)}\left|v(x)\left((f\circ\psi_m)\circ\psi_{n-m}\right)^{(k)}(x)\right|\\
		&&\quad +\sup_{x\in\psi_{-(n-m-1)}\left( [a,\psi(a)]\right)}\left|v(x)\left((f\circ\psi_{m+1})\circ\psi_{n-m-1}\right)^{(k)}(x)\right|.\nonumber
	\end{eqnarray}
	Setting $g=f\circ\psi_m$, in case of $k\geq 1$, for the first summand of the above right hand side, we conclude with Fa\`{a} di Bruno's formula
	\begin{eqnarray*}
		&&\sup_{x\in\psi_{-(n-m)}\left( [a,\psi(a)]\right)}\left|v(x)\left((f\circ\psi_m)\circ\psi_{n-m}\right)^{(k)}(x)\right|\\
		&=&\sup_{x\in\psi_{-(n-m)}\left( [a,\psi(a)]\right)}\left|\sum_{\substack{i_1,i_2,\ldots, i_k\geq 0 \\ i_1+\cdots +ki_k=k}}\frac{k!}{i_1!\cdots i_k!} g^{(i_1+\cdots+i_k)}\left(\psi_{n-m}(x)\right)v(x)\prod_{r=1}^k\left(\frac{\psi_{n-m}^{(r)}(x)}{r!}\right)^{i_r}\right|\\
		&\leq&\max_{\substack{y\in[a,\psi(a)],\\ 0\leq j\leq k}}|g^{(j)}(y)|\sum_{\substack{i_1,i_2,\ldots, i_k\geq 0 \\ i_1+\cdots +ki_k=k}}\frac{k!}{i_1!\cdots i_k!} \sup_{x\in\psi_{-(n-m)}\left( [a,\psi(a)]\right)}\left|v(x)\prod_{r=1}^k\left(\frac{\psi_{n-m}^{(r)}(x)}{r!}\right)^{i_r}\right|,
	\end{eqnarray*}
	which, by the hypotheses on $\psi$ combined with the fact for all $s\in\mathbb{N}$ the function $v\in\mathscr{S}(\R)$ can be written as a product of $s$ functions from $\mathscr{S}(\mathbb{R})$ (see \cite{MR0741451}), tends to zero as $n$ goes to infinity. In case of $k=0$ it follows
	\begin{eqnarray*}
		&&\sup_{x\in\psi_{-(n-m)}\left( [a,\psi(a)]\right)}\left|v(x)\left((f\circ\psi_m)\circ\psi_{n-m}\right)(x)\right|\\
		&\leq&\max_{y\in[a,\psi(a)]}|g(y)|\sup_{x\in\psi_{-(n-m)}\left( [a,\psi(a)]\right)}\left|v(x)\right|,
	\end{eqnarray*}
	which clearly converges to zero as $n$ goes to infinity since $v\in\mathscr{S}(\R)$. 
	
	In the same way one proves that the second summand in the right hand side of \eqref{eq: mixing 1} converges to zero when $n$ tends to infinity which implies
	\[
	\lim_{n\to \infty} \sup_{x\in\mathbb{R}}\left|v(x)(f\circ\psi_n)^{(k)}(x)\right|=0
	\] 
	for every $k\geq 0$, i.e.\ $(f\circ\psi_n)_{n\in \N}$
	converges to zero in $\mathscr{O}_M(\R)$. That $(f\circ\psi_{-n})_{n\in \N}$
	converges to zero in $\mathscr{O}_M(\R)$, too, is proved along the same lines.
	\\
	$(i)\Rightarrow (ii)$ That $\psi$ is injective with a non-vanishing derivative and without fixed points follows from Proposition \ref{proposition: necessary for transitivity}. In order to prove the rest of the properties from $(ii)$, let $a\in\R$ be arbitrary. We proceed by induction with respect to $k$. In what follows, we consider only the case $\psi(x)>x$ for every $x\in\R$. In case of $\psi(x)<x$ for every $x\in\R$, one only has to replace $[a,\psi(a)]$ by $[\psi(a),a]$ in the arguments below.
	Let $v\in \mathscr{S}(\R)$ and $\varepsilon>0$ be arbitrary.
	The sets
	\[
	U=\{f\in \mathscr{O}_M(\R):|f'(x)|>1 \text{ for }x\in [a,\psi(a)]\}
	\]
	and
	\[
	V=\{f\in \mathscr{O}_M(\R): \sup_{x\in\mathbb{R}}
	\left|v(x)f'(x)\right|<\varepsilon\}
	\]
	are non-empty and  open in $\mathscr{O}_M(\R)$. Since $C_\psi$ is mixing, there exists  $N\in\mathbb{N}$ with 
	\[
	C_\psi^n(U)\cap V\not=\emptyset
	\quad\text{and}\quad C_\psi^n(V)\cap U\not=\emptyset
	\quad\text{for every}\quad n\geq N.
	\]
	Let $n\geq N$. There are $f,g\in U$ with 
	$f\circ\psi_n\in V$ and  $g\circ\psi_{-n}\in V$. 
	We have
	\begin{align*}
		\sup_{x\in\psi_{-n}\left([a,\psi(a)]\right)}
		\left|v(x)\psi_n'(x)\right|
		\leq  &  
		\sup_{x\in\psi_{-n}\left([a,\psi(a)]\right)}
		\left|v(x)f'(\psi_n(x))\psi_n'(x)\right| 
		<\varepsilon
	\end{align*}
	and
	\begin{align*}
		\sup_{x\in\psi_{n}\left([a,\psi(a)]\right)}
		\left|v(x)\psi_{-n}'(x)\right|
		\leq  &  
		\sup_{x\in\psi_{n}\left([a,\psi(a)]\right)}
		\left|v(x)g'(\psi_{-n}(x))\psi_{-n}'(x)\right|
		<\varepsilon.
	\end{align*}
	This shows that the condition in $(ii)$ holds for $k=1$.
	
	Assume now that the condition in $(ii)$ holds 
	up to  $k-1$. To finish the induction, for arbitrary $v\in \mathscr{S}(\mathbb{R})$ we have to show
	\[
	\lim_{n\to\infty} \sup_{x\in\psi_{-n}\left([a,\psi(a)]\right)}
	\left|v(x)(\psi_n)^{(k)}(x)\right|=0
	\text{ and }
	\lim_{n\to\infty} \sup_{x\in\psi_{n}\left([a,\psi(a)]\right)}
	\left|v(x)(\psi_{-n})^{(k)}(x)\right|=0.
	\] 
	We will show the second assertion, the first is proved in a similar way.
	
	Let $\varepsilon>0$ be arbitrary,
	\begin{equation*}
		U=\left\lbrace f\in \mathscr{O}_M(\R): 1<\left|f^{(l)}(x)\right|<M \textrm{ for  } x\in [a,\psi(a)],~   0\leq l\leq k\right\rbrace, 
	\end{equation*}
	where
	\[
	M=2\max_{0\leq i\leq k}\frac{(k+1)!(\psi(a)-a+2)^{k+1-i}}{(k+1-i)!}, 
	\]
	and 
	\[V=\left\lbrace
	f\in \mathscr{O}_M(\R):\sup_{x\in\R} \left|
	v(x) f^{(k)}(x)\right|<\frac{\varepsilon}{2}\right\rbrace.\]
	It is clear that $U$ and $V$ are open and non-empty (the polynomial $(x-a+2)^{k+1}$ is in  $U$). 
	Since $C_\psi$ is mixing, there is $N$ such that for all $n\geq N$ we have $C_\psi^n(V)\cap U\not=\emptyset$. Let $n\geq N$. There is 
	$f\in U$ with $f\circ\psi_{-n}\in V$.
	Because $f\in U$, 
	\begin{align*}
		\sup_{x\in\psi_{n}\left([a,\psi(a)]\right)}
		\left|v(x)(\psi_{-n})^{(k)}(x)\right|
		\leq  &  
		\sup_{x\in\psi_{n}\left([a,\psi(a)]\right)}
		\left|v(x)f'(\psi_{-n}(x))(\psi_{-n})^{(k)}(x)\right|
	\end{align*}
	and, by Fa\`{a} di Bruno's formula,
	\begin{align*}
		\sup_{x\in\psi_{n}\left([a,\psi(a)]\right)}&
		\left|v(x)\cdot \left(f'(\psi_{-n}(x))(\psi_{-n})^{(k)}(x)- 
		(f\circ \psi_{-n})^{(k)}(x)\right)\right|\\
		\leq   \sup_{x\in\psi_{n}\left([a,\psi(a)]\right)}&
		|v(x)|\cdot \sum_{\substack{i_1,i_2,\ldots, i_{k-1}\geq 0 \\ i_1+2i_2+\cdots +(k-1)i_{k-1}=k}}
		\frac{k!\cdot M}{i_1!i_2!\cdots i_{k-1}!}\cdot
		\prod_{j=1}^{k-1}\left(\frac{\left|\psi_{-n}^{(j)}(x)\right|}{j!}\right)^{i_j}.
	\end{align*}
	
	Since for all $s\in\mathbb{N}$ every function from $\mathscr{S}(\mathbb{R})$  can be written as a product of $s$ functions from 
	$\mathscr{S}(\mathbb{R})$ (see \cite{MR0741451}), the above and the inductive hypothesis imply that 
	for $n$ large enough 
	\begin{equation*}
		\sup_{x\in\psi_{n}\left([a,\psi(a)]\right)}
		\left|v(x)\cdot \left(f'(\psi_{-n}(x))(\psi_{-n})^{(k)}(x)- 
		(f\circ \psi_{-n})^{(k)}(x)\right)\right|<\frac{\varepsilon}{2}.
	\end{equation*}
	For $n$ large enough,
	because  $f\circ\psi_{-n}\in V$, we have
	\begin{align*}
		\sup_{x\in\psi_{n}\left([a,\psi(a)]\right)}
		\left|v(x)(f\circ\psi_{-n})^{(k)}(x)\right|<\frac{\varepsilon}{2}.
	\end{align*}
	Altogether the above shows for large enough $n$
	$$
	\sup_{x\in\psi_{n}\left([a,\psi(a)]\right)}
	\left|v(x)(\psi_{-n})^{(k)}(x)\right|<\varepsilon
	$$
	which completes the proof.
\end{proof}

Combining Corollary \ref{theorem: sufficient for hypercyclicity} with Theorem \ref{thm: mixing_bij} we obtain the following result.

\quad

\begin{cor}\label{corollary: sufficient for hypercyclicity and mixing form bijective symbol}
	Let $\psi\in \mathscr{O}_M(\R)$ be a bijective function with a non-vanishing derivative and without fixed points such that $\{(\psi_n)':\,n\in\mathbb{Z}\}$ is bounded in $\mathscr{O}_M(\R)$. Then, the composition operator $C_\psi:\mathscr{O}_M(\R)\rightarrow\mathscr{O}_M(\R)$ is sequentially hypercyclic and mixing.
\end{cor} 

\begin{proof}
	By hypothesis, for every $k\in\mathbb{N}$ there are $C>0$ and $m\in\mathbb{N}_0$ such that $|(\psi_n)^{(k)}(x)|\leq C(1+|x|^2)^m$ for every $x\in\R$ and $n\in\mathbb{Z}$. We consider only the case that $\psi(x)>x$. The case $\psi(x)<x$ is proved along the same lines. Hence, $\lim_{n\rightarrow\infty}\psi_n(a)=\infty$ and $\lim_{n\rightarrow\infty}\psi_{-n}(a)=-\infty$ for each $a\in\R$ so that
	$$\lim_{n\rightarrow\infty}\sup_{x\in\psi_{-n}([a,\psi(a)])}|v(x)(\psi_n)^{(k)}(x)|\leq C \lim_{n\rightarrow\infty}\sup_{x\in\psi_{-n}([a,\psi(a)])}|v(x)|(1+|x|^2)^m=0$$
	as well as
	$$\lim_{n\rightarrow\infty}\sup_{x\in\psi_{n}([a,\psi(a)])}|v(x)(\psi_{-n})^{(k)}(x)|\leq C \lim_{n\rightarrow\infty}\sup_{x\in\psi_{n}([a,\psi(a)])}|v(x)|(1+|x|^2)^m=0$$
	for every $v\in\mathscr{S}(\R)$, $k\in\mathbb{N}$. Hence, $C_\psi$ is mixing by Theorem \ref{thm: mixing_bij}.
\end{proof}

\quad

\begin{example}
	From Corollary \ref{corollary: sufficient for hypercyclicity and mixing form bijective symbol} it easily follows that for every $\beta\not=0$
	the operator $C_\psi$, where $\psi(x)=x+\beta$, is mixing on $\mathscr{O}_M(\R)$. This has already been proved in \cite[Proposition 3.6]{MR4419539}.
\end{example}

\quad

\begin{example}
	\label{ex: not_mi}
	Let $\widetilde{\psi}:[0,1]\to\mathbb{R}$ be a smooth function 
	such that $\widetilde{\psi}(x)=3x+1$ for $x\in [0,1/7]$, $\widetilde{\psi}(x)=3x-1$ for $x\in [6/7,1]$
	and $\widetilde{\psi}'(x)>0$ for $x\in [0,1]$ (such a function exists by \cite[Lemma 9]{MR4198421}).
	The function $\psi:\R\to\R$ defined by the formula
	\[
	\psi(x)=\widetilde{\psi(}x-n)+n \quad\text{if}\quad x\in [n,n+1],n\in\mathbb{Z},   
	\]
	belongs to $\mathscr{O}_M(\R)$, has no fixed points and a non-vanishing derivative. One can easily calculate that for every 
	$n\in\mathbb{N}$
	\[
	\psi_{-n}(0)=-n\quad\text{and}\quad(\psi_n)'(\psi_{-n}(0))=\psi_n'(-n)=3^n.
	\]
	Let now $v\in  \mathscr{S}(\mathbb{R})$ be such that 
	$v(x)=e^x$ for $x<0$.
	Then 
	\[
	\lim_{n\to\infty} v(\psi_{-n}(0))(\psi_n)'(\psi_{-n}(0))=\infty.
	\]
	Thus, by Theorem \ref{thm: mixing_bij}, the operator $C_\psi$ is not mixing on 
	$\mathscr{O}_M(\R)$. However it is mixing when considered as an operator acting on $C^\infty(\mathbb{R})$ by \cite[Theorem 4.2]{MR3543813}.
\end{example}

\quad

In order to give more examples we will need the following technical lemma.

\quad

\begin{lem}\label{lem: majorant}
	Let $f\in C^\infty(\mathbb{R})$ be such that $\sup_{x\in\R}(1+|x|^2)^n|f(x)|<\infty$ for every $n\in\N$.
	Then, there exists $g\in\mathscr{S}(\mathbb{R})$, 
	non-decreasing on $(-\infty,0]$ and  non-increasing on $[0,\infty)$, such that $|f(x)|\leq g(x)$ for all $x\in\mathbb{R}$. 
\end{lem}

\begin{proof}
	We set $s_0=0$ and for every $n\in\N$ let 
	\[
	s_n=\sup_{|x|\geq n-1}|f(x)|.
	\]
	One easily verifies that $(s_n)_{n\in\mathbb{N}}$ is a rapidly decreasing sequence. Let $\varphi\colon [0,1]\to\mathbb{R}$ be a smooth function which is equal to $1$ in a neighborhood of 0, equal to $0$ in a neighborhood of $1$, and is  non-increasing on  $[0,1]$.  
	We define $g\colon\mathbb{R}\to\mathbb{R}$ by the formula
	\[
	g(x)=\begin{cases}
		s_{n+1}+(s_n-s_{n+1})\varphi(x-n), & x\in [n,n+1)\text{ for some } n\in\N\cup\{0\};\\
		g(-x), &x<0. 
	\end{cases}
	\]
	It is clear that $g$ has all the requested properties.
\end{proof}

\begin{example}
	\label{ex: mix}
	Let 
	\[
	\widetilde{\psi}(x)=\begin{cases}
		\sqrt{x^2+1}, & x\geq 1,\\
		\frac{\sqrt{2}}{2}x+1, & x\in [-\sqrt{2},0],\\
		-\sqrt{x^2-1}, & x\leq -\sqrt{3},
	\end{cases}
	\]
	and let $\psi$ be any smooth extension of $\widetilde{\psi}$ to $\mathbb{R}$ which satisfies $\psi'(x)>0$ for all $x\in\mathbb{R}$ (such an extension exists by \cite[Lemma 9]{MR4198421}). 
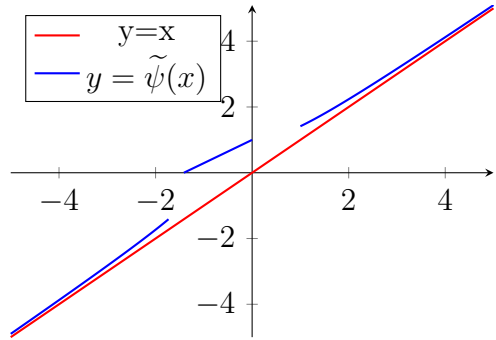
	\begin{center}
		\begin{tikzpicture}
			\begin{axis}[height=6cm,
				width=8cm,
				legend pos=north west,
				axis lines=center,
				]
				\addplot[
				domain = -5:5,
				samples = 200,
				smooth,
				thick,
				red,
				] {x}; 
				\addplot[
				domain = 1:5,
				samples = 200,
				smooth,
				thick,
				blue,
				] {(x^2+1)^(1/2)}; 
				\addplot[
				domain = -5:(-3^(0.5)),
				samples = 200,
				smooth,
				thick,
				blue,
				] {-(x^2-1)^(1/2)}; 
				\addlegendentry{y=x}
				
				\addplot[
				domain = (-2^(0.5)):0,
				samples = 200,
				smooth,
				thick,
				blue,
				] {1/2*2^(0.5)*x+1}; 
				\addlegendentry{$y=\widetilde{\psi}(x)$}
				
			\end{axis}   
		\end{tikzpicture}
	\end{center}
	It is clear that $\psi\in \mathscr{O}_M(\R)$. We will show that the composition operator $C_\psi$ is mixing on $\mathscr{O}_M(\R)$.
	
	In what follows we will use the following properties of the function  $\psi$:
	\begin{enumerate}[label=(\arabic*)]
		\item it is bijective and $\psi(x)>x$ for every $x\in\mathbb{R}$;
		\item for every $x\in\mathbb{R}$  
		\[
		\lim_{n\to\infty} \psi_n(x)=\infty \text{ and }
		\lim_{n\to\infty} \psi_{-n}(x)=-\infty;
		\]
		\item for every $x\leq -\sqrt{2}$ and  
		$n\in\mathbb{N}$ we have $\psi_{-n}(x)=-\sqrt{x^2+n}$;
		\item for every $k\in\mathbb{N}$,  $a\leq b\leq -\sqrt 2$, $v\in \mathscr{S}(\mathbb{R})$ 
		\[
		\lim_{n\to\infty}
		\sup_{x\in [\psi_{-n}(a),\psi_{-n}(b)]}
		\left|v(x)(\psi_n)^{(k)}(x)\right|=0;
		\]
		\item for every $x\geq 1$ and every 
		$n\in\mathbb{N}$ we have $\psi_{n}(x)=\sqrt{x^2+n}$;
		\item for every $k\in\mathbb{N}$, $1\leq a\leq b$, $v\in \mathscr{S}(\mathbb{R})$  
		\[
		\lim_{n\to\infty}
		\sup_{x\in [\psi_{n}(a),\psi_{n}(b)]}
		\left|v(x)(\psi_{-n})^{(k)}(x)\right|=0.
		\]
	\end{enumerate}
	Properties $(1)$, $(2)$, $(3)$ and $(5)$ are easy to verify, we will show now
	that $(4)$ is satisfied, $(6)$ can be checked in a similar way.
	
	In order to prove that $\psi$ satisfies $(4)$ we can assume (in view of Lemma \ref{lem: majorant}) that 
	$v$ is non-negative and non-decreasing on $(-\infty,0]$. For every $n\in\mathbb{N}$ and  $x\in [\psi_{-n}(a),\psi_{-n}(b)]$
	we have $\psi_n(x)=-\sqrt{x^2-n}$. By Fa\`{a} di Bruno's formula, for every $k\geq 1$,  $n\in\mathbb{N}$ and   $x\in [\psi_{-n}(a),\psi_{-n}(b)]$ we thus have
	\begin{align*}
		(\psi_n)^{(k)}(x)=\sum_{\substack{i_1,i_2\geq 0 \\ i_1+2i_2=k}}C_{k,i_1,i_2} (x^2-n)^{-i_1-i_2+1/2}x^{i_1}, 
	\end{align*}
	where the constants $C_{k,i_1,i_2}$ do not depend on $n$. Therefore
	\begin{align*}
		&\sup_{x\in [\psi_{-n}(a),\psi_{-n}(b)]}
		\left|v(x)(\psi_n)^{(k)}(x)\right|
		\stackrel{(3)}{=}
		\sup_{x\in [-\sqrt{a^2+n},-\sqrt{b^2+n}]}
		\left|v(x)(\psi_n)^{(k)}(x)\right|\\
		\leq &
		\sum_{\substack{i_1,i_2\geq 0 \\ i_1+2i_2=k}}
		|C_{k,i_1,i_2}|v(-\sqrt{b^2+n})b^{-2i_1-2i_2+1}\left(\sqrt{a^2+n}\right)^{i_1}.
	\end{align*}
	Since $v\in\mathscr{S}(\R)$
	we get that 
	\[
	\lim_{n\to\infty}
	\sup_{x\in [\psi_{-n}(a),\psi_{-n}(b)]}
	\left|v(x)(\psi_n)^{(k)}(x)\right|=0.
	\]
	By $(4)$ and $(6)$, $\psi$ satisfies condition (iii) of Theorem \ref{thm: mixing_bij} so that $C_\psi$ is mixing on $\mathscr{O}_M(\R)$.
	
\end{example}

\quad

We continue with the analogue to Theorem \ref{thm: mixing_bij} for non-surjective symbol.

\quad

\begin{theorem}
	\label{thm:mix_not_bij}
	Let $\psi\in \mathscr{O}_M(\R)$ be non-surjective. Then, the following conditions are equivalent.
	\begin{enumerate}[label=(\roman*)]
		\item The operator $C_\psi\colon\mathscr{O}_M(\mathbb{R})\to \mathscr{O}_M(\mathbb{R})$, $f\mapsto f \circ \psi$ is mixing.
		\item $\psi$ is injective with a non-vanishing derivative and without fixed points such that for every $a\in\R$ and each $k\in\mathbb{N}$, for arbitrary $v\in \mathscr{S}(\mathbb{R})$ it holds
		\[
		\lim_{n\to\infty} \sup_{x\in\psi_n\left([\min\{a,\psi(a)\},\max\{a,\psi(a)\}]\right)}
		\left|v(x)(\psi_{-n})^{(k)}(x)\right|=0.
		\]
		\item $\psi$ is injective with a non-vanishing derivative and without fixed points, and there is $a\in\R$ such that for each $k\in\mathbb{N}$ and for arbitrary $v\in \mathscr{S}(\mathbb{R})$ it holds
		\[
		\lim_{n\to\infty} \sup_{x\in\psi_n\left([\min\{a,\psi(a)\},\max\{a,\psi(a)\}]\right)}
		\left|v(x)(\psi_{-n})^{(k)}(x)\right|=0.
		\]
	\end{enumerate}
\end{theorem}
\begin{proof}
	Obviously, $(ii)$ implies $(iii)$. The implication $(i)\Rightarrow (ii)$ is shown exactly as in the proof of Theorem \ref{thm: mixing_bij}.
	To prove that $(iii)\Rightarrow (i)$ let us fix non-empty and open sets $U$ and $V$ in $\mathscr{O}_M(\R)$ and  two compactly supported smooth functions $f\in U$ and $g\in V$. We first consider the case $\psi(x)>x$ for every $x\in\R$.
	
	We set $\alpha=-1+\inf\text{supp}\,g$. Moreover, for $n\in\N$ we define
	\[g_n(x)=\begin{cases}
		g(\psi_{-n}(x)), & x\in\psi_n(\mathbb{R}),\\
		0, &\text{otherwise};
	\end{cases}\]
	so that $g_n$ is a compactly supported smooth function, supported in $\psi_n((\alpha,\infty))$. For $n$ large enough we have
	$$\psi_n((\alpha,\infty))\subset (a,\infty)=\cup_{m\in\N_0}(\psi_m(a),\psi_{m+2}(a)).$$ Obviously, the sequence of open intervals $(\psi_m(a),\psi_{m+2}(a))_{m\in\N_0}$ is a locally finite cover of $(a,\infty)$. Let $(\phi_m)_{m\in\N_0}$ be a partition of unity on $(a,\infty)$ subordinate to this cover. For large enough $N$, as in the proof of Theorem \ref{thm: mixing_bij}, one shows that $(\phi_m g_n)_{n\in\N, n\geq N}$ converges to zero in $\mathscr{O}_M(\R)$ for every $m\in\N_0$ which implies that the same holds for $(g_n)_{n\in\N, n\geq N}$. Therefore $f+g_n\in U$ for $n$ large enough.
	Furthermore, since $\psi_{n+1}(\R)\subset\psi_n(\mathbb{R})$ for $n\in\mathbb{N}$ and $\displaystyle \cap_{n\in\mathbb{N}}\psi_n(\mathbb{R})=\emptyset$, we have that $C_\psi^n(f)=0$ for $n$ large enough. Therefore, for $n$ large enough
	we have $C_\psi^n(f+g_n)=g\in V$.
	
	In case $\psi(x)<x$, we define $\tilde{\alpha}=1+\sup\text{supp}\,g$. Replacing $(\alpha,\infty)$ by $(-\infty,\tilde{\alpha})$ and $(a,\infty)$ by $(-\infty,a)$, respectively, the proof is mutatis mutandis the same.
\end{proof}
\begin{example}
	Let 
	\[
	\widetilde{\psi}(x)=\begin{cases}
		e^x, & x\leq 0;\\
		2x, & x\geq 1,
	\end{cases}
	\]
	and let $\psi$ be any smooth extension of $\widetilde{\psi}$ to $\mathbb{R}$ which satisfies $\psi'(x)>0$ for all $x\in\mathbb{R}$ (such an extension exists by \cite[Lemma 9]{MR4198421}). 
	It is clear that $\psi\in \mathscr{O}_M(\R)$ and $\psi(x)>x$ for all $x\in \R$.  Obviously $\psi_n(x)=2^nx$ whenever $x\geq 1$ and for $x\in\psi_n([1,\psi(1)])$ we have $\psi_{-n}(x)=2^{-n}x$. It is straightforward to show that condition $(iii)$ in Theorem \ref{thm:mix_not_bij} is fulfilled for $a=1$. Therefore the composition operator $C_\psi$ is mixing on $\mathscr{O}_M(\R)$.
\end{example}

\section{A relation to Abel's equation}\label{section: Abel}
In this section we relate the mixing property of composition operators acting on $\mathscr{O}_M(\mathbb{R})$ with the solvability of Abel's functional equation, i.e. the equation 
\[
H(\psi(x))=H(x)+1
\]
where $\psi\colon\R\to\R$ is a given function.
Solvability of this equation  is well-understood in various situations. For example it is known that if $\psi\colon\mathbb{R}\to\mathbb{R}$ is a bijective smooth (or real analytic) function with no fixed points, then this equation has a smooth (real analytic) solution, see \cite{MR1488145,MR3323563}. 

\quad

\begin{theorem}
	\label{thm: abel}
	Let $\psi\in \mathscr{O}_M(\mathbb{R})$ be bijective.
	The following conditions are equivalent.
	\begin{enumerate}[label=(\roman*)]
		\item The operator $C_\psi\colon\mathscr{O}_M(\mathbb{R})\to \mathscr{O}_M(\mathbb{R})$, $f\mapsto f \circ \psi$ is mixing and for every $v\in \mathscr{S}(\mathbb{R})$
		\[
		\lim_{n\to\infty}v(\psi_n(0))\cdot n=0
		\text{ and }
		\lim_{n\to\infty}v(\psi_{-n}(0))\cdot n=0.
		\]
		\item There exists  $H\in \mathscr{O}_M(\R)$
		with a non-vanishing derivative and which satisfies the equation \[ H(\psi(x))=H(x)+1\text{ for every }x\in \mathbb{R}.\]
	\end{enumerate}
\end{theorem}
\begin{proof}$ $\\ 
	$(i)\Rightarrow (ii)$
	Since $C_\psi$ is mixing, by Proposition \ref{proposition: necessary for transitivity}, the function $\psi$ has no fixed points and has a non-vanishing derivative. In what follows we will assume that $\psi(x)>x$ for every $x\in\R$, the other case can be done in a similar way.  By \cite[Thm. 8]{MR4198421} there exists a bijective smooth function $H$ with a non-vanishing derivative and which satisfies the equation 
	\begin{equation}
		\label{abel}
		H(\psi(x))=H(x)+1\text{ for every }x\in \mathbb{R}.
	\end{equation}
	We need to show that $H\in \mathscr{O}_M(\R)$, i.e.\ that for every
	$v\in\mathscr{S}(\mathbb{R})$ and $k\geq 0$
	\begin{equation}
		\label{todo}
		\sup_{x\in\mathbb{R}}\left|v(x)H^{(k)}(x)\right|<\infty.
	\end{equation}
	In view of Lemma \ref{lem: majorant} we may assume that $v$ is
	non-decreasing on $(-\infty,0)$ and  non-increasing on $[0,\infty)$.
	To prove $(\ref{todo})$ it is enough to show that 
	\begin{equation}
		\label{om1}
		\lim_{n\to\infty} \sup_{x\in [\psi_{n}(0),\psi_{n+1}(0)]}
		\left|
		v(x)H^{(k)}(x)\right|=0
	\end{equation}
	and
	\begin{equation}
		\label{om2}
		\lim_{n\to\infty} \sup_{x\in [\psi_{-n-1}(0),\psi_{-n}(0)]}
		\left|
		v(x)H^{(k)}(x)\right|=0.
	\end{equation}
	We will show $(\ref{om1})$, the proof of $(\ref{om2})$ is similar. 
	
	From $(\ref{abel})$ it follows that for $n\in\mathbb{N}$
	and $x\in [\psi_{n}(0),\psi_{n+1}(0)]$ we have 
	\[
	H(x)=H(\psi_{-n}(x))+n \text{ and }H^{(k)}(x)=(H\circ\psi_{-n})^{(k)}(x)
	\text{ for } k\in\mathbb{N}.\]
	Thus
	\begin{align*}
		\sup_{x\in [\psi_{n}(0),\psi_{n+1}(0)]}
		\left|
		v(x)H(x)\right|=&
		\sup_{x\in [\psi_{n}(0),\psi_{n+1}(0)]}
		\left|
		v(x)(H(\psi_{-n}(x))+n)\right|\\
		\leq &
		\left(n+\sup_{x\in [0,\psi(0)]}|H(x)|\right)\cdot v(\psi_n(0))
	\end{align*}
	and
	\begin{align*}
		\sup_{x\in [\psi_{n}(0),\psi_{n+1}(0)]}
		\left|
		v(x)H^{(k)}(x)\right|=&
		\sup_{x\in [\psi_{n}(0),\psi_{n+1}(0)]}
		\left|
		v(x)(H\circ\psi_{-n})^{(k)}(x)\right|.
	\end{align*}
	Therefore $(\ref{om1})$ follows from the assumptions on $\psi$
	and Theorem $\ref{thm: mixing_bij}$ combined with Fa\`{a} di Bruno's formula and the boundedness of $H^{(j)}$ on $[0,\psi(0)]$, $j\in\N\cup\{0\}$.

	$(ii)\Rightarrow (i)$ Due to $H(\psi(x))=H(x)+1$ for every $x\in\R$ it follows that $\psi$ does not have fixed points. Additionally, since $H(\psi_n(x))=H(x)+n$ for every $x\in\R$, $n\in\mathbb{Z}$, we conclude that $\lim_{n\rightarrow\infty}H(\psi_n(0))=\infty$ and $\lim_{n\rightarrow\infty}H(\psi_{-n}(0))=-\infty$ which implies the surjectivity of $H$. Since $H$ has non-vanishing derivative we conclude that $H$ is bijective. Moreover, for arbitrary $v\in\mathscr{S}(\R)$ we have
	$$\lim_{n\rightarrow\infty}v\left(\psi_n(0)\right)n=\lim_{n\rightarrow\infty}v\left(\psi_n(0)\right)\left(H\left(\psi_n(0)\right)-H(0)\right)=0$$
	because $Hv, H(0)v\in\mathscr{S}(\R)$ and $\lim_{n\rightarrow\infty}|\psi_n(0)|=\infty$, the latter since $\psi$ has no fixed points. In the same way it follows $\lim_{n\rightarrow\infty}v\left(\psi_{-n}(0)\right)n=0$.
	
	The conditions in $(ii)$ 
	imply that the diagram
	\[
	\begin{CD}
		\mathscr{O}_M(\mathbb{R}) @>C_{x+1}>> \mathscr{O}_M(\mathbb{R})\\ 
		@VC_H VV  @VC_H VV \\
		\mathscr{O}_M(\mathbb{R}) @>C_{\psi}>> \mathscr{O}_M(\mathbb{R})
	\end{CD}
	\]
	commutes and that the operator $C_H$ has dense range 
	(since all compactly supported smooth functions are in its image because $H$ is bijective). Thus $C_\psi$ is quasi-conjugate to the mixing operator $C_{x+1}$ and hence mixing.
\end{proof}

\begin{cor}
	\label{cor: hypercyclic}
	If $\psi\in \mathscr{O}_M(\R)$ is bijective and satisfies the conditions of Theorem \ref{thm: abel} (i), then $C_\psi$ is quasi-conjugate to the operator $C_{x+1}$ on  $\mathscr{O}_M(\mathbb{R})$. Therefore it is hypercyclic and chaotic.  
\end{cor}

\quad

\begin{remark} It is not clear to the authors if for  every mixing 
	composition operator $C_\psi$ on $\mathscr{O}_M(\R)$, where $\psi$ is bijective, it automatically holds
	\begin{equation}
		\label{notknow}
		\lim_{n\to\infty}v(\psi_n(0))\cdot n=0
		\text{ and }
		\lim_{n\to\infty}v(\psi_{-n}(0))\cdot n=0
		\text{ for any }v\in\mathscr{S}(\mathbb{R}).
	\end{equation}
	If this was the case, then every mixing $C_\psi$ would already be hypercyclic and chaotic by the above corollary.
	
	Condition $(\ref{notknow})$ is satisfied whenever there is $\beta>0$ for which $\psi(x)>x+\beta$ for every $x\in\R$. Example \ref{ex: not_mi} shows that the latter is not a sufficient condition for mixing.
	Example \ref{ex: mix} shows that $\lim_{x\to\infty}(\psi(x)-x)=0$ may happen for a mixing $C_\psi$. 
\end{remark}

\quad

\noindent\textbf{Open problems.} Let $\psi\in\mathscr{O}_M(\R)$ be such that $C_\psi$ is mixing on $\mathscr{O}_M(\R)$.
\begin{itemize}
	\item[1.] Assume additionally that $\psi$ is bijective. Is it true that
	$$\lim_{n\to\infty}v(\psi_n(0))\cdot n=0
	\text{ and }
	\lim_{n\to\infty}v(\psi_{-n}(0))\cdot n=0$$
	holds for every $v\in\mathscr{S}(\mathbb{R})$?
	\item[2.] Is $C_\psi$ (sequentially) hypercyclic on $\mathscr{O}_M(\R)$?
\end{itemize}

\quad

While we do not know the answer to problem 1, the next theorem shows that the sequence $(\psi_n(0))_{n\in\N}$ cannot grow too slowly.

\quad

\begin{theorem}
	\label{not-trans}
	Let $f\in C^\infty(\mathbb{R})$ be real valued such that $\sup_{x\in\R}(1+|x|^2)^n|f(x)|<\infty$ for every $n\in\N$, $\inf_{x\in\R}|1+f'(x)|>0$, and $f'\in\mathscr{O}_M(\R)$. Then, for $\psi(x)=x+f(x)$, the operator $C_\psi\colon\mathscr{O}_M(\mathbb{R})\to \mathscr{O}_M(\mathbb{R})$ is not topologically transitive.
\end{theorem}

\quad

It should be noted that under the hypotheses of the above theorem $\psi\in\mathscr{O}_M(\R)$ so that $C_\psi:\mathscr{O}_M(\R)\rightarrow\mathscr{O}_M(\R)$ is correctly defined. To prove the above theorem we need the following lemma
which is of independent interest.

\quad

\begin{lem}\label{prop: inverse}
	Let $f\in C^\infty(\R)$ be real valued such that $\sup_{x\in\R}(1+|x|^2)|f(x)|<\infty$, $\inf_{x\in\R}|1+f'(x)|>0$, and $f'\in\mathscr{O}_M(\R)$. Let $\psi(x)=x+f(x)$, $x\in\R$. Then, $\psi$ is bijective and $g\circ\psi^{-1}\in\mathscr{S}(\R)$ for every $g\in\mathscr{S}(\R)$.
\end{lem}

\begin{proof}
	By hypothesis, $\psi'(x)\neq 0$ so that $\psi$ is injective. Moreover, since obviously $\lim_{|x|\rightarrow\infty}|f(x)|=0$, it follows $\lim_{x\rightarrow\pm\infty}\psi(x)=\pm\infty$ so that $\psi$ is bijective. Additionally, for $|x|$ sufficiently large we have
	$$|\psi(x)|\leq|x|+|f(x)|\leq|x|+\frac{\sup_{y\in\R}(1+|y|^2)|f(y)|}{(1+|x|^2)}\leq 2|x|$$
	for $x$ large which implies
	\begin{equation}\label{eq: symbol part 1}
		|\psi^{-1}(x)|\geq \frac{|x|}{2}\geq|x|^{1/k}
	\end{equation}
	whenever $|x|\geq k$ for some suitable $k\in\N$.
	Obviously, $|\psi^{-1}(x)|\leq (1+|\psi^{-1}(x)|^2)^{1/2}$, and due to \cite{MR1790920}, for $n\in\N$ there is a polynomial $P_n$ in $n$ variables with integer coefficients such that
	\begin{eqnarray*}
		(\psi^{-1})^{(n)}(x)&=&\left(\frac{1}{\psi'(\psi^{-1}(x))}\right)^{2n-1}P_n\left(\psi'(\psi^{-1}(x)),\psi^{(2)}(\psi^{-1}(x)),\ldots,\psi^{(n)}(\psi^{-1}(x))\right)\\
		&=&\left(\frac{1}{1+f'(\psi^{-1}(x))}\right)^{2n-1}\times\\
		&&\quad\times P_n\left(1+f'(\psi^{-1}(x)),f^{(2)}(\psi^{-1}(x)),\ldots,f^{(n)}(\psi^{-1}(x))\right).
	\end{eqnarray*}
	In particular, since $f'\in\mathscr{O}_M(\R)$, for a suitable constant $C>0$ and $k\in\N$ it holds for arbitrary $x\in\R$
	$$|(\psi^{-1})^{(n)}(x)|\leq \left(\frac{1}{\inf_{y\in\R}|1+f'(y)|}\right)^{2n-1}C\left(1+|\psi^{-1}(x)|^2\right)^k.$$
	Combining this with \eqref{eq: symbol part 1}, an application of \cite[Theorem 2.3]{MR3763350} proves the claim.
\end{proof}




\begin{proof}[Proof of Theorem \ref{not-trans}]
	The inclusion $\mathscr{O}_M(\R)\hookrightarrow C^\infty(\R)$ is
	continuous and has dense range, therefore  topological 
	transitivity of $C_\psi$ on $\mathscr{O}_M(\R)$ implies that $C_\psi$ is also topologically transitive on $C^\infty(\R)$. Thus 
	by   \cite[Theorem 4.2]{MR3543813} if $\psi$ has a fixed 
	point or $\psi'(x)=0$ for some $x\in\R$, then $C_\psi$ is not topologically transitive.
	
	From now on we will assume that $\psi'(x)\not=0$ for every $x\in\R$ and that $\psi$ has no fixed point which implies $f(x)\neq 0$ for every $x\in\R$. Moreover, we assume that $f(x)>0$ for all $x\in \R$; the proof in case $f(x)<0$ for all $x\in \R$ is similar.
	
	Let $g\in\mathscr{S}(\R)$ be as in Lemma \ref{lem: majorant} for $f$. By Lemma \ref{prop: inverse} we have  $g\circ\psi^{-1} \in\mathscr{S}(\R)$.
	Therefore, both sets 
	\[
	U=\{u\in \mathscr{O}_M(\R): 0<u(0)<1 \text{ and } 2<u(\psi(0))<3\}
	\] 
	and
	\[
	V= \{v\in \mathscr{O}_M(\R): \sup_{x\in\mathbb{R}}|v'(x)g(\psi^{-1}(x))|<0.5\}
	\] 
	are open in $\mathscr{O}_M(\R)$ and non-empty. If $v\in V$ and $n\geq 1$, then by the Mean Value theorem we get that 
	\[
	\left|C_\psi^n(v)(\psi(0))-C_\psi^n(v)(0)\right|=
	\left|v(\psi_{n+1}(0))-v(\psi_n(0))\right|
	=|v'(\xi)f(\psi_n(0))|,
	\]
	where $\xi\in [\psi_n(0),\psi_{n+1}(0)]$.
	Using monotonicity of $g\circ\psi^{-1}$ we get that  
	\begin{align*}
		|v'(\xi)f(\psi_n(0))|\leq &|v'(\xi)g(\psi_n(0))|\\
		=&|v'(\xi)(g\circ\psi^{-1})(\psi_{n+1}(0))|\\
		\leq &|v'(\xi)(g\circ\psi^{-1})(\xi)|<0.5.
	\end{align*}
	Thus $C_\psi^n(v)\not\in U$ and therefore $C_\psi$ is not topologically transitive.
\end{proof}
\begin{example}\label{ex: not_tt} Obviously, the function $\psi(x)=x+\exp(-x^2/2), x\in\R,$ belongs to $\mathscr{O}_M(\R)$, has no fixed points and satisfies $\psi'(x)\neq 0$, $x\in\R$. While the corresponding composition operator $C_\psi$ is hypercyclic/topologically transitive/mixing on $C^\infty(\R)$ by \cite[Theorem 4.2]{MR3543813}, it is not topologically transitive on $\mathscr{O}_M(\R)$ by Theorem \ref{not-trans}.
\end{example}

\quad

\noindent\textbf{Acknowledgement.} The research on the topic of this article was initiated during a visit of the first named author at Adam Mickiewicz University. He is very grateful to his colleagues from Pozna\'n for the cordial hospitality during this stay. Additionally, we thank the anonymous referees for their suggestions and comments which helped to improve the article.

\end{document}